\documentclass{amsart}

\newtheorem{theorem}{Theorem}[section]
\newtheorem{lemma}[theorem]{Lemma}

\newtheorem{corollary}[theorem]{Corollary}
\theoremstyle{definition}

\theoremstyle{remark}

\numberwithin{equation}{section}

\begin{document}
\title[Schatten $p$-norm inequalities]{Schatten $p$-norm inequalities
related to a characterization of inner product spaces}
\author[O. Hirzallah, F. Kittaneh, M.S. Moslehian]{Omar Hirzallah$^1$,\ Fuad
Kittaneh$^2$\ and\ Mohammad Sal Moslehian$^3$}
\address{$^1$Department of Mathematics, Hashemite University, Zarqa, Jordan.}
\email{o.hirzal@hu.edu.jo}
\address{$^2$Department of Mathematics, University of Jordan, Amman, Jordan.}
\email{fkitt@ju.edu.jo}
\address{$^3$Department of Pure Mathematics, Ferdowsi University of Mashhad, P.O. Box 1159, Mashhad 91775, Iran;
\newline Center of Excellence in Analysis on
Algebraic Structures (CEAAS), Ferdowsi University of Mashhad, Iran.}
\email{moslehian@ferdowsi.um.ac.ir and moslehian@ams.org}

\begin{abstract}
Let $A_1, \cdots A_n$ be operators acting on a separable complex Hilbert
space such that $\sum_{i=1}^n A_i=0$. It is shown that if $A_1, \cdots A_n$
belong to a Schatten $p$-class, for some $p>0$, then
\begin{equation*}
2^{p/2}n^{p-1} \sum_{i=1}^n \|A_i\|^p_p \leq \sum_{i,j=1}^n\|A_i\pm A_j\|^p_p
\end{equation*}
for $0<p\leq 2$, and the reverse inequality holds for $2\leq p<\infty$.
Moreover,
\begin{equation*}
\sum_{i,j=1}^n\|A_i\pm A_j\|^2_p \leq 2n^{2/p} \sum_{i=1}^n \|A_i\|^2_p
\end{equation*}
for $0<p\leq 2$, and the reverse inequality holds for $2\leq p<\infty$.
These inequalities are related to a characterization of inner product spaces
due to E.R. Lorch.
\end{abstract}

\maketitle


\section{Introduction}

Let $B({\mathcal{H}})$ be the algebra of all bounded linear
operators on a separable complex Hilbert space ${\mathcal{H}}$. Let
$A\in B({\mathcal{H}})$ be compact, and let $0 < p < \infty$. The
Schatten $p$-norm (quasi-norm) for $1 \leq p < \infty \,(0< p < 1)$
is defined by $\|A\|_p=(\text{tr} |A|^p)^{1/p}$, where $\text{tr}$
is the usual trace functional and $|A|=(A^*A)^{1/2}$. Clearly
$\left\|\,|A|^2\,\right\|_{p/2}=\|A\|_p^2$ for $p>0$. For $p>0$, the
Schatten $p$-class, denoted by $C_p$, is defined to be the set of
those compact operators $A$ for which $\|A\|_p$ is finite. When
$p=2$, the Schatten $p$-norm $\|A\|_2=(\text{tr} |A|^2)^{1/2}$ is
called the Hilbert--Schmidt norm of $A$. For $p>0$, $C_p$ is a
two-sided ideal in $B({\mathcal{H}})$. For $1 \leq p <\infty$, $C_p$
is a Banach space; in particular, if $A_1, \cdots, A_n \in C_p$,
then the triangle inequality for $\|\cdot\|_p$ asserts that
\begin{eqnarray}  \label{1}
\left\|\sum_{i=1}^n A_i\right\|_p \leq \sum_{i=1}^n \|A_i\|_p\,.
\end{eqnarray}
However, for $0<p<1$, the quasi-norm $\|.\|_p$ does not satisfy the
triangle inequality. It has been shown in \cite{B-K1} (see, also,
\cite{B-K2}) that if $A_1, \cdots, A_n \in C_p$ are positive and $0
<p \leq 1$, then
\begin{eqnarray}  \label{2}
\sum_{i=1}^n \|A_i\|_p \leq \left\|\sum_{i=1}^n A_i\right\|_p\,.
\end{eqnarray}
For more information on the theory of the Schatten $p$-classes, the reader
is referred to \cite{BHA, SIM}.

It is well-known that a normed space $X$ is an inner product space if and
only if for every $x, y \in X$, we have
\begin{eqnarray}  \label{5}
\|x+y\|^2+\|x-y\|^2 = 2(\|x\|^2+\|y\|^2)\,.
\end{eqnarray}
The identity (\ref{5}) is known as the parallelogram law. A version of this
fundamental law holds for the Hilbert-Schmidt norm. Generalizations of the
parallelogram law for the Schatten $p$-norms have been given in the form of
the celebrated Clarkson inequalities (see \cite{B-K2} and references
therein). These inequalities have proven to be very useful in analysis,
operator theory, and mathematical physics.

Another known characterization of inner product spaces is due to E.R. Lorch
\cite{LOR}. He proved that a normed space $X$ is an inner product space if
and only if for a fixed integer $n\geq 3$, and $x_1, \cdots, x_n \in X$ with
$\sum_{i=1}^nx_i=0$, we have
\begin{eqnarray}
\sum_{i,j=1}^n\|x_i-x_j\|^2=2n\sum_{i=1}^n\|x_i\|^2.
\end{eqnarray}
Since $C_2$ is a Hilbert space under the inner product $\langle A,
B\rangle=\text{tr} (B^*A)$, it follows that if $A_1, \cdots, A_n \in
C_2$ with $\sum _{i=1}^n A_i=0$, then
\begin{eqnarray}  \label{4}
\sum_{i,j=1}^n\|A_i- A_j\|^2_2=2n \sum_{i=1}^n \|A_i\|^2_2\,.
\end{eqnarray}

In this paper, we establish operator inequalities for the Schatten $p$-norms
that form natural generalizations of the identity (\ref{4}). Our
inequalities presented here seem natural enough and applicable to be widely
useful.


\section{Main results}

To achieve our goal, we need the following lemma (see \cite{B-K1, KIT}).

\begin{lemma}
\label{lem} Let $A_1, \cdots, A_n \in C_p$ for some $p>0$. If $A_1, \cdots,
A_n$ are positive, then

\begin{eqnarray*}
n^{p-1}\sum_{i=1}^n \|A_i\|_p^p \leq \left\|\sum_{i=1}^n A_i \right\|_p^p
\leq \sum_{i=1}^n \|A_i\|_p^p \qquad (\text{a})
\end{eqnarray*}
for $0 < p \leq 1$; and
\begin{eqnarray*}
\sum_{i=1}^n \|A_i\|_p^p \leq \left\|\sum_{i=1}^n A_i \right\|_p^p \leq
n^{p-1} \sum_{i=1}^n \|A_i\|_p^p \qquad (\text{b})
\end{eqnarray*}
for $1 \leq p < \infty$.
\end{lemma}

A commutative version of Lemma \ref{lem} can be formulated for scalars as
follows: If $a_1, \cdots, a_n$ are nonnegative real numbers, then
\begin{eqnarray}  \label{real1}
n^{p-1}\sum_{i=1}^n a_i^p \leq \left(\sum_{i=1}^n a_i \right)^p \leq
\sum_{i=1}^n a_i^p
\end{eqnarray}
for $0 < p \leq 1$; and
\begin{eqnarray}  \label{real2}
\sum_{i=1}^n a_i^p \leq \left(\sum_{i=1}^n a_i \right)^p \leq n^{p-1}
\sum_{i=1}^n a_i^p
\end{eqnarray}
for $1 \leq p < \infty$. These inequalities follow, respectively, from the
concavity of the function $f(t)=t^p,\, t\in[0,\infty)$ for $0 < p \leq 1$,
and the convexity of the function $f(t)=t^p,\, t\in[0,\infty)$ for $1 \leq p
< \infty$.

Our first main result can be stated as follows. It furnishes a
generalization of (\ref{4}).


\begin{theorem}
\label{th1} Let $A_1, \cdots, A_n, B_1, \cdots, B_n \in C_p$, for
some $p>0$, such that $\sum _{i,j=1}^n A_i^*B_j=0$. Then
\begin{equation*}
2^{^{p/2}-1}n^{p-1}\left(\sum_{i=1}^n \|A_i\|_p^p+\sum_{i=1}^n
\|B_i\|_p^p\right) \leq \sum_{i,j=1}^n\|A_i\pm B_j\|^p_p
\end{equation*}

\noindent for $0<p\leq 2$; and
\begin{equation*}
\sum_{i,j=1}^n\|A_i\pm B_j\|^p_p \leq 2^{^{p/2}-1}n^{p-1}\left(\sum_{i=1}^n
\|A_i\|_p^p+\sum_{i=1}^n \|B_i\|_p^p\right)
\end{equation*}

\noindent for $2\leq p< \infty$.
\end{theorem}

\begin{proof}
Let $0 < p \leq 2$. Then
\begin{eqnarray*}
\sum_{i,j=1}^n \|A_i \pm B_j\|_p^p &=& \sum_{i,j=1}^n \left\|\,|A_i \pm
B_j|^2\,\right\|_{p/2}^{p/2} \\
&\geq& \left\|\sum_{i,j=1}^n |A_i \pm B_j|^2\right\|_{p/2}^{p/2} \\
&&\,\,\,\,\,\,\,\,\,\,\, \quad (\text{by the second inequality of
Lemma\, \ref{lem}(a)}) \\
&=& \left\|\sum_{i,j=1}^n \left(|A_i|^2+ |B_j|^2\pm A_i^*B_j \pm
B_j^*A_i\right)\right\|_{p/2}^{p/2} \\
&=& \left\|\sum_{i,j=1}^n |A_i|^2+ \sum_{i,j=1}^n |B_j|^2\right\|_{p/2}^{p/2}
\\
&=& n^{p/2}\left\|\sum_{i=1}^n|A_i|^2+ \sum_{i=1}^n|B_i|^2
\right\|_{p/2}^{p/2} \\
&\geq& (2n)^{{p/2}-1}n^{p/2}\left(\sum_{i=1}^n \|A_i\|_p^p+\sum_{i=1}^n
\|B_i\|_p^p\right) \\
&&\,\,\,\,\,\,\,\,\,\,\, \quad (\text{by the first inequality of
Lemma\, \ref{lem}(a)}) \\
&=& 2^{{p/2}-1}n^{p-1}\left(\sum_{i=1}^n \|A_i\|_p^p+\sum_{i=1}^n
\|B_i\|_p^p\right).
\end{eqnarray*}
This proves the first part of the theorem.

Based on Lemma \ref{lem}(b), one can employ an argument similar to that used
in the proof of the first part of the theorem to prove the second part.
\end{proof}

An application of Theorem \ref{th1} can be seen in the following result,
which is a natural generalization of (\ref{4}).

\begin{corollary}
Let $A_1, \cdots, A_n \in C_p$, for some $p>0$, such that $\sum _{i=1}^n
A_i=0$. Then
\begin{equation*}
2^{p/2}n^{p-1} \sum_{i=1}^n \|A_i\|^p_p \leq \sum_{i,j=1}^n\|A_i\pm A_j\|^p_p
\end{equation*}
for $0<p\leq 2$; and
\begin{equation*}
\sum_{i,j=1}^n\|A_i\pm A_j\|^p_p \leq 2^{p/2}n^{p-1} \sum_{i=1}^n \|A_i\|^p_p
\end{equation*}
for $2\leq p<\infty$. In particular, (\ref{4}) holds in the case where $p=2$.
\end{corollary}

\begin{proof}
Since $\sum _{i=1}^n A_i=0$, we have $\big(\sum _{i=1}^n A_i \big)^*
\big(\sum _{i=1}^n A_i\big)=0$. Consequently, $\sum _{i,j=1}^n
A_i^*A_j=0$. Utilizing Theorem \ref{th1} with $B_j=A_j \quad (1 \leq
j\leq n)$, we obtain the result.
\end{proof}


Using the same reasoning as in the proof of Theorem \ref{th1}, one
can obtain the following result concerning operators having
orthogonal ranges. Recall that ranges of two operators $A, B \in
B({\mathcal{H}})$ are orthogonal (written $\text{ran} A \bot
\text{ran} B$) if and only if $A^*B=0$.

\begin{theorem}
Let $A_1, \cdots, A_n \in C_p$, for some $p>0$, such that $\text{ran} A_i
\bot \text{ran} A_j$ for $i \neq j$. Then
\begin{equation*}
(2n\pm 2)^{p/2}n^{{p/2}-1} \sum_{i=1}^n \|A_i\|^p_p \leq
\sum_{i,j=1}^n\|A_i\pm A_j\|^p_p
\end{equation*}

\noindent for $0<p\leq 2$; and
\begin{equation*}
\sum_{i,j=1}^n \|A_i\pm A_j\|^p_p \leq (2n\pm 2)^{p/2}n^{{p/2}-1}
\sum_{i=1}^n \|A_i\|^p_p
\end{equation*}

\noindent for $2\leq p< \infty$.
\end{theorem}


Our second main result, which also leads to a generalization of (\ref{4}),
can be stated as follows.

\begin{theorem}
\label{th2} Let $A_1, \cdots, A_n, B_1, \cdots, B_n \in C_p$, for
some $p>0$, such that $\sum _{i,j=1}^n A_i^*B_j=0$. Then
\begin{equation*}
\sum_{i,j=1}^n\|A_i\pm B_j\|^2_p \leq n^{2/p}\sum_{i=1}^n
\left\|\left(|A_i|^2+|B_i|^2\right)^{1/2}\right\|_p^2
\end{equation*}

\noindent for $0<p\leq 2$; and
\begin{equation*}
n^{2/p}\sum_{i=1}^n \left\|\left(|A_i|^2+|B_i|^2\right)^{1/2}\right\|_p^2
\leq \sum_{i,j=1}^n\|A_i\pm B_j\|^2_p
\end{equation*}

\noindent for $2\leq p< \infty$.
\end{theorem}

\begin{proof}
Let $0 < p \leq 2$. Then
\begin{eqnarray*}
\sum_{i,j=1}^n \|A_i \pm B_j\|_p^2 &=& \sum_{i,j=1}^n \left\|\,|A_i \pm
B_j|^2\,\right\|_{p/2} \\
&\leq& \left\|\sum_{i,j=1}^n |A_i \pm B_j|^2\right\|_{p/2}\quad
\left(\text{by} (\ref{2})\right) \\
&=& \left\|\sum_{i,j=1}^n \left(|A_i|^2+ |B_j|^2\pm A_i^*B_j \pm
B_j^*A_i\right)\right\|_{p/2} \\
&=& \left\|\sum_{i,j=1}^n |A_i|^2+ \sum_{i,j=1}^n |B_j|^2\right\|_{p/2} \\
&=& n\left\|\sum_{i=1}^n\left(|A_i|^2+ |B_i|^2\right)\right\|_{p/2} \\
&\leq& n\left(\sum_{i=1}^n
\left\||A_i|^2+|B_i|^2\right\|_{p/2}^{p/2}\right)^{2/p} \\
&&\,\,\,\,\,\,\,\,\,\,\, \quad (\text{by the second inequality of
Lemma\, \ref{lem}(a)}) \\
&=& n\left(\sum_{i=1}^n \left\|\left(
|A_i|^2+|B_i|^2\right)^{1/2}\right\|_p^p\right)^{2/p} \\
&\leq& n^{2/p}\sum_{i=1}^n \left\|\left(
|A_i|^2+|B_i|^2\right)^{1/2}\right\|_p^2. \\
&&\,\,\,\,\,\,\,\,\,\,\, \quad (\text{by the second inequality of
(\ref{real2})
}). \\
\end{eqnarray*}
This proves the first part of the theorem.

Based on (\ref{1}), the first inequality of Lemma \ref{lem}(b), and the
first inequality of (\ref{real1}), one can employ an argument similar to
that used in the proof of the first part of the theorem to prove the second
part.
\end{proof}


An application of Theorem \ref{th2}, which is another natural generalization
of (\ref{4}), can be seen as follows.

\begin{corollary}
Let $A_1, \cdots, A_n \in C_p$, for some $p>0$, such that $\sum _{i=1}^n
A_i=0$. Then
\begin{equation*}
\sum_{i,j=1}^n\|A_i\pm A_j\|^2_p \leq 2n^{2/p} \sum_{i=1}^n \|A_i\|^2_p
\end{equation*}
for $0<p\leq 2$; and
\begin{equation*}
2n^{2/p} \sum_{i=1}^n \|A_i\|^2_p \leq \sum_{i,j=1}^n\|A_i\pm A_j\|^2_p
\end{equation*}
for $2\leq p< \infty$. In particular, (\ref{4}) holds in the case
where $p=2$.
\end{corollary}


Using the same reasoning as in the proof of Theorem \ref{th2}, one can
obtain the following result concerning operators having orthogonal ranges.

\begin{theorem}
Let $A_1, \cdots, A_n \in C_p$, for some $p>0$, such that $\text{ran} A_i
\bot \text{ran} A_j$ for $i \neq j$. Then
\begin{equation*}
\sum_{i,j=1}^n \|A_i\pm A_j\|^2_p \leq 2n^{{2/p}-1}(n\pm 1)\sum_{i=1}^n
\|A_i\|^2_p
\end{equation*}
for $0<p\leq 2$; and
\begin{equation*}
2n^{{2/p}-1}(n\pm 1)\sum_{i=1}^n \|A_i\|^2_p \leq \sum_{i,j=1}^n \|A_i\pm
A_j\|^2_p
\end{equation*}
for $2\leq p< \infty$.
\end{theorem}


\end{document}